\documentclass[12pt,reqno]{amsart}

\usepackage{hyperref}
\usepackage[usenames]{color}
\usepackage[latin1]{inputenc} 
\usepackage{float}
\usepackage{bbm}

\usepackage[english]{babel} 

\usepackage{amssymb, amsmath}
\usepackage{geometry,graphicx}
\usepackage{mhchem}

\newcommand{\leb}{\mathrm{L}}

\theoremstyle{plain}
\newtheorem{theorem}{Theorem}[section] 

\newtheorem{example}[theorem]{Example}

\numberwithin{equation}{section} 
\title[Non-autonomous second order linear RODE in Mathematical Modeling]{Some notes to extend the study on random non-autonomous second order linear differential equations appearing in Mathematical Modeling}

\author{J. Calatayud, J.-C. Cort\'{e}s, M. Jornet} 

\begin{document}

\maketitle

\begin{center}
\noindent
\address{Instituto Universitario de Matem\'{a}tica Multidisciplinar,\\
Universitat Polit\`{e}cnica de Val\`{e}ncia,\\
Camino de Vera s/n, 46022, Valencia, Spain\\
email: jucagre@doctor.upv.es; jccortes@imm.upv.es; marjorsa@doctor.upv.es
}
\end{center}

\begin{abstract}
The objective of this paper is to complete certain issues from our recent contribution [J. Calatayud, J.-C. Cort\'es, M. Jornet, L. Villafuerte, \textit{Random non-autonomous second order linear differential equations: mean square analytic solutions and their statistical properties}, Advances in Difference Equations, 2018:392, 1--29 (2018)]. We restate the main theorem therein that deals with the homogeneous case, so that the hypotheses are clearer and also easier to check in applications. Another novelty is that we tackle the non-homogeneous equation with a theorem of existence of mean square analytic solution and a numerical example. We also prove the uniqueness of mean square solution via an habitual Lipschitz condition that extends the classical Picard Theorem to mean square calculus. In this manner, the study on general random non-autonomous second order linear differential equations with analytic data processes is completely resolved. Finally, we relate our exposition based on random power series with polynomial chaos expansions and the random differential transform method, being the latter a reformulation of our random Fr\"obenius method. \\
\\
\textit{Keywords: Random non-autonomous second order linear differential equation; Mean square analytic solution; Random power series; Uncertainty quantification.} \\ 
\\
\textit{AMS Classification 2010: 34F05; 60H10; 60H35; 65C20; 65C30.}
\end{abstract}

\section{Introduction}

The important role played by differential equations in dealing with mathematical modeling is beyond discussion. They are powerful tools to describe the dynamics of phenomena appearing in a variety of distinct realms, such as Engineering, Biomedicine, Chemistry, social behavior, etc. \cite{Goodwine_2011,Schiesser_2014,Brown_2007}. In this paper we concentrate on a class of differential equations that have played a distinguished role in a variety of applications in science, in particular in Physics and Engineering, namely, second order linear differential equations. Indeed, these equations have been successfully applied to describe, for example, vibrations in springs (free undamped or simple harmonic motion, damped vibrations subject to a frictional force or forced vibrations affected by an external force), analysis of electric circuits made up of an electromotive force supplied by a battery or generator, a resistor, an inductor and a capacitor. In the former type of problems, second order linear differential equations are formulated by applying Newton's Second Law, while in the latter case this class of equations appear via the application of Kirchhoff's voltage law. In these examples, the formulation of second order linear differential equations to describe the aforementioned physical  problems appear as direct applications of important laws of Physics. However, this class of equations also arise indirectly when solving significant partial differential equations in Physics. In this regard, Airy, Hermite, Laguerre, Legendre or Bessel differential equations are non-autonomous second order linear differential equations that emerge in this way. For example, Airy equation appears when solving Schr\"{o}dinger's equation with triangular potential and for a particle subject to a one-dimensional constant force field \cite{Vallee_Soares_2004}; Hermite equation emerges in dealing with the analysis of Schr\"{o}dinger's equation for a harmonic oscillator in Quantum Mechanics \cite{Fedoryuk_2001}; Laguerre equation plays a main role in Quantum Mechanics for the study of the hydrogen atom via the Schr\"{o}dinger's equation using radial functions \cite[Ch.~10]{Spain_Smith_1970}; Legendre equation appears when solving the Laplace equation to compute the potential of a conservative field such as the space gravitational potential using spherical coordinates \cite{Johnson_Pedersen_1986}; and, finally, Bessel equation is encountered, for example, when solving Helmholtz equation in cylindrical or spherical coordinates by using the method of separation of variables \cite[Ch.~9]{Spain_Smith_1970}.

In all the previous examples, two important features can be highlighted to motivate our subsequent analysis. First, the coefficients of the differential equations are analytic (specifically polynomials). Second, these coefficients depend upon physical parameters that, in practice, need to be fixed after measurements, therefore they involve uncertainty. Both facts motivate the study of random  non-autonomous second order linear differential equations whose coefficients and initial conditions are   analytic stochastic  processes and random initial conditions, respectively. The aim of our contribution is to advance in the analysis (both theoretical and practical) of this important class of equations.  In our subsequent development, we mainly focus on theoretical aspects of such analysis with the conviction that it can become really useful in future applications where randomness is considered in that class of differential equations. In this sense, some numerical experiments illustrating and demonstrating the potentiality of our main findings are also included. The study of random non-autonomous second order linear differential equations has been carried out for particular cases, such as Airy, Hermite, Legendre, Laguerre and Bessel equations (see \cite{airy,hermite,legendre,arxiv,laguerre,bessel}, respectively), and the general case \cite{Golmankhaneh_2013,Khudair_2011_a,Khudair_2011_b,contribucio}.

For the sake of clarity in the presentation, the layout of the paper is as follows: in Section~\ref{homoSection} we study the homogeneous random non-autonomous second order linear differential. The analysis includes a result (Theorem \ref{nostre2}) that simplifies the application of a recent finding by the authors that is particularly useful in practical cases. This issue is illustrated via several numerical examples where both the random Airy and Hermite differential equations are treated. In Section~\ref{sec_non-homogeneous} the analysis is extended to random non-autonomous second order linear differential equation with forcing term. In this study we have included conditions under which the solution is unique in the mean square sense. This theoretical study is supported with a numerical example too. Section~\ref{sec_comparison} is addressed to enrich our contribution by comparing the random Fr\"{o}benius method proposed in this paper against other alternative approaches widely used in the extant literature, specifically gPC expansions, Monte Carlo simulations and the random differential transform method. Conclusions and future research lines are drawn in Section~\ref{sec_conclusiones}.

\section{Homogeneous case} \label{homoSection}

We consider the general form of a homogeneous random non-autonomous second order linear differential equation in an underlying complete probability space $(\Omega,\mathcal{F},\mathbb{P})$:
\begin{equation}
 \begin{cases} \ddot{X}(t)+A(t)\dot{X}(t)+B(t)X(t)=0, \; t\in\mathbb{R}, \\ X(t_0)=Y_0, \\ \dot{X}(t_0)=Y_1. \end{cases} 
 \label{problem}
\end{equation}
It is assumed that the stochastic processes $A(t)$ and $B(t)$ are analytic at $t_0$ in the mean square sense \cite[p.~99]{soong}:
\[ A(t)=\sum_{n=0}^\infty A_n(t-t_0)^n,\quad B(t)=\sum_{n=0}^\infty B_n(t-t_0)^n, \quad t\in (t_0-r,r_0+r), \]
with convergence in $\leb^2(\Omega)$. The terms $A_0,A_1,\ldots$ and $B_0,B_1,\ldots$ are random variables. In fact, they are related to $A(t)$ and $B(t)$, respectively, via Taylor expansions:
\[ A_n=\frac{A^{(n)}(t_0)}{n!},\quad B_n=\frac{B^{(n)}(t_0)}{n!}, \]
where the derivatives are considered in the mean square sense. According to the Fr\"obenius method, we look for a solution stochastic process $X(t)$ also expressible as a mean square convergent random power series on $(t_0-r,t_0+r)$:
\[ X(t)=\sum_{n=0}^\infty X_n(t-t_0)^n. \]
Mean square convergence is important, as it allows approximating the expectation (average) and variance (dispersion) statistics of $X(t)$ at each $t$ \cite[Th.~4.2.1, Th.~4.3.1]{soong}. This is one of the primary goals of uncertainty quantification \cite{Smith}.

In \cite{contribucio}, after stating and proving some auxiliary theorems on random power series (differentiation of random power series in the $\leb^p(\Omega)$ sense \cite[Th.~3.1]{contribucio}, and Merten's theorem for random series in the mean square sense \cite[Th.~3.2]{contribucio}, which generalize their deterministic counterparts), the main theorem was stated as follows:

\begin{theorem} \label{nostre} \cite[Th.~3.3]{contribucio}
Let $A(t)=\sum_{n=0}^\infty A_n (t-t_0)^n$ and $B(t)=\sum_{n=0}^\infty B_n (t-t_0)^n$ be two random series in the $\leb^2(\Omega)$ setting, for $t\in (t_0-r,t_0+r)$, being $r>0$ finite and fixed. Assume that the initial conditions $Y_0$ and $Y_1$ belong to $\leb^2(\Omega)$. Suppose that there is a constant $C_r>0$, maybe dependent on $r$, such that $\|A_n\|_{\leb^\infty(\Omega)}\leq C_r/r^n$ and $\|B_n\|_{\leb^\infty(\Omega)}\leq C_r/r^n$, $n\geq0$. Then the stochastic process $X(t)=\sum_{n=0}^\infty X_n(t-t_0)^n$, $t\in (t_0-r,t_0+r)$, where
\begin{equation} X_0=Y_0,\quad X_1=Y_1, \label{x0y0} \end{equation}
\begin{equation}
 X_{n+2}=\frac{-1}{(n+2)(n+1)}\sum_{m=0}^n \left[(m+1)A_{n-m}X_{m+1}+B_{n-m}X_m\right],\;n\geq0, 
 \label{xn2}
\end{equation}
is the unique analytic solution to the random initial value problem (\ref{problem}) in the mean square sense. 
\end{theorem}

This theorem is a generalization of the deterministic Fr\"obenius method to a random framework. As it was demonstrated in \cite{contribucio}, Theorem~\ref{nostre} has many applications in practice. It supposes a unified approach to study the most well-known second order linear random differential equations: Airy \cite{airy}, Hermite \cite{hermite}, Legendre \cite{legendre,arxiv}, Laguerre \cite{laguerre} and Bessel \cite{bessel}. The results established in these articles, \cite{airy,hermite,legendre,arxiv,laguerre,bessel}, are particular cases of Theorem~\ref{nostre}. The main reason of why this fact occurs is explained in \cite[Subsection~3.3]{contribucio}: given a random variable $Z$, the fact that its centered absolute moments grow at most exponentially, $\mathbb{E}[|Z|^n]\leq HR^n$ for certain $H>0$ and $R>0$, is equivalent to $Z$ being essentially bounded, $\|Z\|_{\leb^\infty(\Omega)}\leq R$. 

Notice that Theorem~\ref{nostre} does not require any independence assumption on the random input parameters. Moreover, from Theorem~\ref{nostre}, \cite{contribucio} obtains error estimates for the approximation of the solution stochastic process, its mean and its variance.

Let us see that Theorem~\ref{nostre} may be put in an easier to handle form. We substitute the growth condition on the coefficients $A_0,A_1,\ldots$ and $B_0,B_1,\ldots$ by the $\leb^\infty(\Omega)$ convergence of the random power series that define $A(t)$ and $B(t)$. In this manner, in practical applications, one does not need to find any constant $C_r$, see the forthcoming Examples~\ref{ex1}, \ref{ex2}, \ref{ex3} and \ref{ex4}.

\begin{theorem} \label{nostre2} 
Let $A(t)=\sum_{n=0}^\infty A_n (t-t_0)^n$ and $B(t)=\sum_{n=0}^\infty B_n (t-t_0)^n$ be two random series in the $\leb^\infty(\Omega)$ setting, for $t\in (t_0-r,t_0+r)$, being $r>0$ finite and fixed. Assume that the initial conditions $Y_0$ and $Y_1$ belong to $\leb^2(\Omega)$. Then the stochastic process $X(t)=\sum_{n=0}^\infty X_n(t-t_0)^n$, $t\in (t_0-r,t_0+r)$, whose coefficients are defined by (\ref{x0y0})--(\ref{xn2}), is the unique analytic solution to the random initial value problem (\ref{problem}) in the mean square sense. 
\end{theorem}
\begin{proof}
By \cite[Lemma~2.3]{arxiv}, 
\[ \sum_{n=0}^\infty \|A_n\|_{\leb^\infty(\Omega)} |t-t_0|^n<\infty, \quad \sum_{n=0}^\infty \|B_n\|_{\leb^\infty(\Omega)} |t-t_0|^n<\infty, \]
for $t\in (t_0-r,t_0+r)$. Thus, for each $0\leq r_1<r$, 
\[ \sum_{n=0}^\infty \|A_n\|_{\leb^\infty(\Omega)} r_1^n<\infty, \quad \sum_{n=0}^\infty \|B_n\|_{\leb^\infty(\Omega)} r_1^n<\infty. \]
Since the sequences $\{\|A_n\|_{\leb^\infty(\Omega)} r_1^n\}_{n=0}^\infty$ and $\{\|B_n\|_{\leb^\infty(\Omega)} r_1^n\}_{n=0}^\infty$ tend to $0$, they are both bounded by a number $C_{r_1}>0$:
\[ \|A_n\|_{\leb^\infty(\Omega)}\leq \frac{C_{r_1}}{r_1^n},\quad \|B_n\|_{\leb^\infty(\Omega)}\leq \frac{C_{r_1}}{r_1^n},\quad n\geq0. \]
Then Theorem~\ref{nostre} is applicable with $r_1$: the stochastic process $X(t)=\sum_{n=0}^\infty X_n(t-t_0)^n$ whose coefficients are given by (\ref{x0y0})--(\ref{xn2}) is a mean square solution to (\ref{problem}) on $(t_0-r_1,t_0+r_1)$. Now, since $r_1$ is arbitrary, we can can extend this result to the whole interval $(t_0-r,r_0+r)$.
\end{proof}

Notice that we have proved that Theorem~\ref{nostre} from \cite{contribucio} entails Theorem~\ref{nostre2}. But the other way around also holds: Theorem~\ref{nostre2} implies Theorem~\ref{nostre}. Thus, both theorems are equivalent and offer the same information. Indeed, if we assume the hypotheses from Theorem~\ref{nostre}, then 
\[ \|A_n\|_{\leb^\infty(\Omega)}r_1^n \leq C_r\left(\frac{r_1}{r}\right)^n,\quad \|B_n\|_{\leb^\infty(\Omega)}\leq C_r\left(\frac{r_1}{r}\right)^n, \]
for any $0\leq r_1<r$, and since $\sum_{n=0}^\infty (\frac{r_1}{r})^n<\infty$, by comparison we derive that 
\[ \sum_{n=0}^\infty \|A_n\|_{\leb^\infty(\Omega)}r_1^n<\infty,\quad \sum_{n=0}^\infty \|B_n\|_{\leb^\infty(\Omega)}r_1^n<\infty, \]
which entails that the series of $A(t)=\sum_{n=0}^\infty A_n (t-t_0)^n$ and $B(t)=\sum_{n=0}^\infty B_n (t-t_0)^n$ converge $\leb^\infty(\Omega)$, by \cite[Lemma~2.3]{arxiv}, for $t\in (t_0-r_1,t_0+r_1)$. As $r_1$ is arbitrary, the $\leb^\infty(\Omega)$ convergence holds for $t\in (t_0-r,t_0+r)$. This is exactly the hypothesis used in Theorem~\ref{nostre2}.

Let us see that Theorem~\ref{nostre2} has an easier to handle form by checking the hypotheses in the examples from \cite{contribucio}.

\begin{example} \label{ex1}
Airy's random differential equation is defined as follows:
\begin{equation}
 \begin{cases} \ddot{X}(t)+AtX(t)=0, \; t\in\mathbb{R}, \\ X(0)=Y_0, \\ \dot{X}(0)=Y_1, \end{cases} 
 \label{problem_airy}
\end{equation}
where $A$, $Y_0$ and $Y_1$ are random variables. We suppose that $Y_0$ and $Y_1$ have centered second order absolute moments. In \cite{airy}, the hypothesis used in order to obtain a mean square analytic solution $X(t)$ is $\mathbb{E}[|A|^n]\leq HR^n$, $n\geq n_0$. Notice that this hypothesis is equivalent to $\|A\|_{\leb^\infty (\Omega)}\leq R$, by \cite[Subsection~3.3]{contribucio}. In our general notation, $A(t)=0$ and $B(t)=At$. Due to the boundedness of the random variable $A$, the $\leb^\infty(\Omega)$ convergence of the series that define $A(t)$ and $B(t)$ holds, so Theorem~\ref{nostre2} (and Theorem~\ref{nostre}) is applicable: there is an analytic solution stochastic process $X(t)$ to (\ref{problem_airy}) on $\mathbb{R}$.
\end{example}

\begin{example} \label{ex2}
Hermite's random differential equation is given as follows:
\begin{equation}
 \begin{cases} \ddot{X}(t)-2t\dot{X}(t)+AX(t)=0, \; t\in\mathbb{R}, \\ X(0)=Y_0, \\ \dot{X}(0)=Y_1, \end{cases} 
 \label{problem_hermite}
\end{equation}
where $A$, $Y_0$ and $Y_1$ are random variables. We suppose that $Y_0,Y_1\in\leb^2(\Omega)$. In \cite{hermite}, the hypothesis utilized to derive a mean square analytic solution $X(t)$ is $\mathbb{E}[|A|^n]\leq HR^n$, $n\geq n_0$. This hypothesis is equivalent to $\|A\|_{\leb^\infty (\Omega)}\leq R$, by \cite[Subsection~3.3]{contribucio}. Under boundedness of the random variable $A$, the input stochastic processes $A(t)=-2t$ and $B(t)=A$ are expressible as $\leb^\infty(\Omega)$ convergent random power series. Hence, both Theorem~\ref{nostre2} and Theorem~\ref{nostre} are applicable, and guarantee the existence of a mean square solution process $X(t)$ on $\mathbb{R}$.
\end{example}

\begin{example} \label{ex3}
We consider the following random linear differential equation with polynomial data processes:
\begin{equation}
 \begin{cases} \ddot{X}(t)+(A_0+A_1t)\dot{X}(t)+(B_0+B_1t)X(t)=0, \; t\in\mathbb{R}, \\ X(0)=Y_0, \\ \dot{X}(0)=Y_1. \end{cases} 
 \label{problem_polynomial}
\end{equation}
If the initial conditions $Y_0$ and $Y_1$ belong to $\leb^2(\Omega)$ and the random input parameters $A_0$, $A_1$, $B_0$ and $B_1$ are bounded random variables, then the hypotheses of Theorem~\ref{nostre2} (and Theorem~\ref{nostre}) fulfill and we derive that there is a mean square solution process $X(t)$ on $\mathbb{R}$.
\end{example}

\begin{example} \label{ex4}
We consider (\ref{problem}) with non-polynomial analytic stochastic process:
\begin{equation}
 \begin{cases} \ddot{X}(t)+A(t)\dot{X}(t)+B(t)X(t)=0, \; t\in\mathbb{R}, \\ X(0)=Y_0, \\ \dot{X}(0)=Y_1, \end{cases} 
 \label{problem_series}
\end{equation}
where $A_n\sim\text{Beta}(11,15)$, for $n\geq0$, $B_n=1/n^2$, for $n\geq1$, and $Y_0,Y_1\in\leb^2(\Omega)$. Since $\sum_{n=0}^\infty \|A_n\|_{\leb^\infty(\Omega)} |t|^n=\sum_{n=0}^\infty |t|^n$ is finite for $t\in (-1,1)$, and analogously for $B(t)$, Theorem~\ref{nostre2} (and consequently Theorem~\ref{nostre}) implies that there is a mean square solution $X(t)$ to (\ref{problem_series}) on $(-1,1)$.
\end{example}

We raise the following open problem, which would imply that the hypotheses used in Theorem~\ref{nostre2} are necessary: ``If there exists a point $t_1\in (t_0-r,t_0+r)$ such that $A(t_1)\notin\leb^\infty(\Omega)$ or $B(t_1)\notin\leb^\infty(\Omega)$, then there exist two initial conditions $Y_0,Y_1\in\leb^2(\Omega)$ such that (\ref{problem}) has no mean square solution on $(t_0-r,t_0+r)$''. Although we have not been able to prove this statement (which might be false), we think that the proof might be based on the reasoning used in \cite[Example, pp.~4--5]{strand}. 

This open problem, despite being of theoretical interest, does not contribute in practical applications. In numerical experiments, one usually truncates the stochastic processes $A(t)$ and $B(t)$ (that is, work with a partial sum instead of the whole Taylor series). This is not uncommon when dealing with stochastic systems computationally, as one requires a dimensionality reduction of the problem. If the coefficients of $A(t)$ and/or $B(t)$ have unbounded support, one may truncate them so that the hypotheses of Theorem~\ref{nostre} and Theorem~\ref{nostre2} fulfill, and the probabilistic behavior of the data processes does not change much. 

\section{Non-homogeneous case}\label{sec_non-homogeneous}

In this section, we generalize (\ref{problem}) by adding a stochastic source term:
\begin{equation}
 \begin{cases} \ddot{X}(t)+A(t)\dot{X}(t)+B(t)X(t)=C(t), \; t\in\mathbb{R}, \\ X(t_0)=Y_0, \\ \dot{X}(t_0)=Y_1. \end{cases} 
 \label{problem2}
\end{equation}
This new term $C(t)$ is analytic at $t_0$ in the mean square sense \cite[p.~99]{soong}, with Taylor series
\[ C(t)=\sum_{n=0}^\infty C_n (t-t_0)^n,\quad t\in (t_0-r,t_0+r). \]
The coefficients $C_0,C_1,\ldots$ are random variables. For this new model (\ref{problem2}), we want to find conditions under which $X(t)$ is an analytic mean square solution on $(t_0-r,t_0+r)$. This work was not done in \cite{contribucio}, and it completes the study on the random non-autonomous second order linear differential equation with analytic input processes.

The following theorem is a generalization of Theorem~\ref{nostre2}:

\begin{theorem} \label{nostre3} 
Let $A(t)=\sum_{n=0}^\infty A_n (t-t_0)^n$ and $B(t)=\sum_{n=0}^\infty B_n (t-t_0)^n$ be two random series in the $\leb^\infty(\Omega)$ setting, for $t\in (t_0-r,t_0+r)$, being $r>0$ finite and fixed. Let $C(t)=\sum_{n=0}^\infty C_n (t-t_0)^n$ be a random series in the mean square sense on $(t_0-r,t_0+r)$. Assume that the initial conditions $Y_0$ and $Y_1$ belong to $\leb^2(\Omega)$. Then the stochastic process $X(t)=\sum_{n=0}^\infty X_n(t-t_0)^n$, $t\in (t_0-r,t_0+r)$, whose coefficients are defined by 
\begin{equation} X_0=Y_0,\quad X_1=Y_1, \label{x0y0_2} \end{equation}
\begin{equation}
 X_{n+2}=\frac{1}{(n+2)(n+1)}\left\{-\sum_{m=0}^n \left[(m+1)A_{n-m}X_{m+1}+B_{n-m}X_m\right]+C_n\right\},\;n\geq0, 
 \label{xn2_2}
\end{equation}
is the unique analytic solution to the random initial value problem (\ref{problem2}) in the mean square sense. 
\end{theorem}
\begin{proof}
After performing formal manipulations with the random power series (all of them justified by \cite[Th.~3.1, Th.~3.2]{contribucio}), one finds an analogous expression to \cite[expr.~(9)]{contribucio}:
\small
\[ \sum_{n=0}^\infty \left[(n+2)(n+1)X_{n+2}+\sum_{m=0}^n\left(A_{n-m}(m+1)X_{m+1}+B_{n-m}X_{m}\right)\right](t-t_0)^n=\sum_{n=0}^\infty C_n(t-t_0)^n. \]
\normalsize
From here, the recursive relation (\ref{xn2_2}) is obtained. 

Thus, it only remains to prove that the random power series $X(t)=\sum_{n=0}^\infty X_n(t-t_0)^n$, whose coefficients are defined by (\ref{x0y0_2}) and (\ref{xn2_2}), converges in the mean square sense.

From the hypothesis $Y_0,Y_1\in\leb^2(\Omega)$ and by induction on $n$ in expression (\ref{xn2_2}), we obtain that $X_n\in\leb^2(\Omega)$ for all $n\geq0$. On the other hand, by \cite[Lemma~2.3]{arxiv}, 
\[ \sum_{n=0}^\infty \|A_n\|_{\leb^\infty(\Omega)} s^n<\infty, \;\; \sum_{n=0}^\infty \|B_n\|_{\leb^\infty(\Omega)} s^n<\infty,\;\; \sum_{n=0}^\infty \|C_n\|_{\leb^2(\Omega)} s^n<\infty, \]
for $0<s<r$. As the general term of a convergent series tends to $0$, we have the following bounds:
\begin{equation} \|A_n\|_{\leb^\infty(\Omega)}\leq \frac{D_s}{s^n},\;\; \|B_n\|_{\leb^\infty(\Omega)}\leq \frac{D_s}{s^n},\;\; \|C_n\|_{\leb^2(\Omega)}\leq \frac{D_s}{s^n},\;n\geq0, 
 \label{anbncn}
\end{equation}
for certain constant $D_s>0$ that depends on $s$. Then, from (\ref{xn2_2}), if we apply $\leb^2(\Omega)$ norms and (\ref{anbncn}), we obtain:
\small
\begin{align}
\|X_{n+2}\|_{\leb^2(\Omega)}\leq {} & \frac{1}{(n+2)(n+1)}\left\{\sum_{m=0}^n \left[(m+1)\|A_{n-m}X_{m+1}\|_{\leb^2(\Omega)}+\|B_{n-m}X_m\|_{\leb^2(\Omega)}\right]+\|C_n\|_{\leb^2(\Omega)}\right\} \nonumber \\
\leq {} &  \frac{1}{(n+2)(n+1)}\frac{D_s}{s^n}\left\{\sum_{m=0}^n s^m \left((m+1)\|X_{m+1}\|_{\leb^2(\Omega)}+\|X_m\|_{\leb^2(\Omega)}\right)+1\right\}. \label{X_n_2}
\end{align}
\normalsize
Define $H_0:=\|Y_0\|_{\leb^2(\Omega)}$, $H_1:=\|Y_1\|_{\leb^2(\Omega)}$ and
\begin{equation}
 H_{n+2}:=\frac{1}{(n+2)(n+1)}\frac{D_s}{s^n}\left\{\sum_{m=0}^n s^m \left((m+1)H_{m+1}+H_m\right)+1\right\},\quad n\geq0.
 \label{hn2_2}
\end{equation}
From (\ref{X_n_2}) and (\ref{hn2_2}), by induction on $n$ it is trivially seen that $\|X_n\|_{\leb^2(\Omega)}\leq H_n$, for $n\geq0$. If we check that $\sum_{n=0}^\infty H_n s^n<\infty$, then the random series that defines $X(t)$ converges in the mean square sense, as wanted.

We rewrite (\ref{hn2_2}) so that $H_{n+2}$ is expressed as a function of $H_{n+1}$ and $H_n$ (second order recurrence equation). By assuming $n\geq1$, we perform the following operations:
\begin{align}
H_{n+2}= {} & \frac{1}{(n+2)(n+1)}\frac{D_s}{s^n}\left(\sum_{m=0}^{n-1} s^m \left((m+1)H_{m+1}+H_m\right)+1+s^n \left((n+1)H_{n+1}+H_n\right)\right) \nonumber \\
= {} & \frac{1}{(n+2)(n+1)}\frac{D_s}{s^n}\frac{(n+1)n}{D_s}s^{n-1} \bigg(\underbrace{\frac{1}{(n+1)n}\frac{D_s}{s^{n-1}}\left\{\sum_{m=0}^{n-1} s^m \left((m+1)H_{m+1}+H_m\right)+1\right\}}_{=H_{n+1}}\bigg) \nonumber \\
+ {} & \frac{D_s}{n+2}H_{n+1}+\frac{D_s}{(n+2)(n+1)}H_n \nonumber \\
= {} & \left(\frac{n}{(n+2)s}+\frac{D_s}{n+2}\right)H_{n+1}+\frac{D_s}{(n+2)(n+1)}H_n. \label{curteta2}
\end{align}
This difference equation of order $2$ has as initial conditions 
\[ H_2=\frac{D_s}{2}\left(H_1+H_0+1\right),\quad H_1=\|Y_1\|_{\leb^2(\Omega)},\quad H_0=\|Y_0\|_{\leb^2(\Omega)}. \]
Notice that $H_2$ is obtained from \eqref{hn2_2}. Expression~(\ref{curteta2}) coincides with \cite[expr.~(12)]{contribucio} (although with different initial conditions). Then the method of proof for $\sum_{n=0}^\infty H_n s^n<\infty$ is identical to the last part of the proof of \cite[Th.~3.3]{contribucio}.
\end{proof}

\begin{example} \label{exNou}
Let us consider an Hermite's random differential equation with a stochastic source term:
\begin{equation}
 \begin{cases} \ddot{X}(t)-2t\dot{X}(t)+AX(t)=C t^2, \; t\in\mathbb{R}, \\ X(0)=Y_0, \\ \dot{X}(0)=Y_1, \end{cases} 
 \label{problem_hermite2}
\end{equation}
where $A$, $C$, $Y_0$ and $Y_1$ are random variables. Due to the non-homogeneity of the equation, this example cannot be addressed with \cite{contribucio}. We have set the following probability distributions:
\[ A\sim\text{Bernoulli}(0.35),\quad Y_0\sim\text{Gamma}(2,2),\quad (Y_1,C)\sim\text{Multinomial}(3,\{0.2,0.8\}) \]
(for the Gamma distribution, we use the shape-rate notation), where $A$, $Y_0$ and $(Y_1,C)$ are independent. Notice that we are considering both discrete and absolutely continuous random variables/vectors, and also both independent and non-independent random variables/vectors. Thus, Fr\"obenius method covers a wide variety of situations in practice. Since $A$ is bounded and $Y_0,Y_1,C\in\leb^2(\Omega)$, Theorem~\ref{nostre3} ensures that the random power series $X(t)=\sum_{n=0}^\infty X_n(t-t_0)^n$ defined by (\ref{x0y0_2})--(\ref{xn2_2}) is a mean square solution to (\ref{problem_hermite2}) on $\mathbb{R}$. By considering the partial sums $X^N(t)=\sum_{n=0}^N X_n(t-t_0)^n$, we approximate the expectation and variance of $X(t)$ as 
\[ \mathbb{E}[X(t)]=\lim_{N\rightarrow\infty}\mathbb{E}[X_N(t)],\quad \mathbb{V}[X(t)]=\lim_{N\rightarrow\infty}\mathbb{V}[X_N(t)]. \]
The computations have been performed in the software Mathematica\textsuperscript{\tiny\textregistered}. Our code to build the partial sum $X^N(t)$ has been the following one:
\footnotesize
\begin{verbatim}
X[n_?NonPositive] := Y0;
X[1] = Y1;
X[n_] := 1/(n*(n - 1))*(-Sum[(m + 1)*A[n - 2 - m]*X[m + 1] + 
        B[n - 2 - m]*X[m], {m, 0, n - 2}] + CC[n - 2]);
seriesX[t_, t0_, N_] := X[0] + Sum[X[n]*(t - t0)^n, {n, 1, N}];
\end{verbatim}
\normalsize
For each numeric value of \verb|N|, the functions $t\mapsto \mathbb{E}[X^N(t)]$ and $t\mapsto \mathbb{V}[X^N(t)]$ have been calculated with the built-in function \verb|Expectation| applied to \verb|seriesX[t, 0, N]| (with symbolic \verb|t|), by setting the desired probability distributions to \verb|A[n]|, \verb|B[n]| and \verb|CC[n]|. In Table~\ref{exampleTable1} and Table~\ref{exampleTable2}, we show $\mathbb{E}[X^N(t)]$ and $\mathbb{V}[X^N(t)]$ for $N=19$, $N=20$ and $0\leq t\leq 1.5$. Both orders of truncation produce similar results, which agrees with the theoretical convergence. Observe that, as we move away from the initial condition $t_0=0$, larger orders of truncation are needed. This indicates that Fr\"obenius method might be computationally inviable for large $t$. The results have been compared with Monte Carlo simulation (with $100,000$ and $200,000$ realizations). 

\begin{table}[hbtp!]
\begin{center}
\begin{tabular}{|c|c|c|c|c|} \hline
$t$ & $\mathbb{E}[X^{19}(t)]$ & $\mathbb{E}[X^{20}(t)]$ & MC $100,000$ & MC $200,000$  \\ \hline
$0.00$ & $1$ & $1$ & $0.995893$ & $1.00266$  \\ \hline
$0.25$ & $1.14231$ & $1.14231$ & $1.13899$ & $1.14544$  \\ \hline
$0.50$ & $1.28890$ & $1.28890$ & $1.28672$ & $1.29236$ \\ \hline
$0.75$ & $1.49183$ & $1.49183$ & $1.49130$ & $1.49547$ \\ \hline
$1.00$ & $1.85892$ & $1.85892$ & $1.86087$ & $1.86246$ \\ \hline
$1.25$ & $2.62573$ & $2.62574$ & $2.63173$ & $2.62863$ \\ \hline
$1.50$ & $4.34772$ & $4.34784$ & $4.36111$ & $4.34892$ \\ \hline
\end{tabular}
\caption{Approximation of $\mathbb{E}[X(t)]$ with $N=19$, $N=20$ and Monte Carlo simulations. Example \ref{exNou}.}
\label{exampleTable1}
\end{center}
\end{table}

\begin{table}[hbtp!]
\begin{center}
\begin{tabular}{|c|c|c|c|c|} \hline
$t$ & $\mathbb{V}[X^{19}(t)]$ & $\mathbb{V}[X^{20}(t)]$ & MC $100,000$ & MC $200,000$  \\ \hline
$0.00$ & $0.5$ & $0.5$ & $0.493124$ & $0.504501$\\ \hline
$0.25$ & $0.520298$ & $0.520298$ & $0.514702$ & $0.524803$\\ \hline
$0.50$ & $0.597008$ & $0.597008$ & $0.593603$ & $0.601376$  \\ \hline
$0.75$ & $0.790556$ & $0.790556$ & $0.790161$ & $0.794549$ \\ \hline
$1.00$ & $1.27425$ & $1.27425$ & $1.27702$ & $1.27759$ \\ \hline
$1.25$ & $2.60694$ & $2.60694$ & $2.60987$ & $2.60982$  \\ \hline
$1.50$ & $6.94095$ & $6.94100$ & $6.92663$ & $6.94787$ \\ \hline
\end{tabular}
\caption{Approximation of $\mathbb{V}[X(t)]$ with $N=19$, $N=20$ and Monte Carlo simulations. Example \ref{exNou}.}
\label{exampleTable2}
\end{center}
\end{table}

\end{example}

Notice that the theoretical error estimates from \cite[Subsection~3.6]{contribucio} apply in this case too, since all estimates rely on the majorization $\|X_n\|_{\leb^2(\Omega)}\leq H_n$ and the recursive equation (\ref{curteta2}), which also hold in \cite{contribucio}.

An important issue that was not treated in the recent contribution \cite{contribucio} is the uniqueness of mean square solution. To deal with uniqueness, we use an habitual extension of the classical Picard Theorem to mean square calculus \cite[Th.~5.1.2]{soong}, see Theorem~\ref{uniquenessTh}. Notice that, in our setting of analyticity for $A(t)$ and $B(t)$ in the $\leb^\infty(\Omega)$ sense, one has that $A(t)$ and $B(t)$ are continuous in the $\leb^\infty(\Omega)$, so the uniqueness from Theorem~\ref{uniquenessTh} is applicable.

\begin{theorem} \label{uniquenessTh}
If $A(t)$ and $B(t)$ are continuous stochastic processes in the $\leb^{\infty}(\Omega)$ sense, then the mean square solution to (\ref{problem2}) is unique.
\end{theorem}
\begin{proof}
We write (\ref{problem}) as a first-order random differential equation, which is the setting under study in \cite{soong}:
\[ \underbrace{\begin{pmatrix} \dot{X}(t) \\ \ddot{X}(t) \end{pmatrix}}_{\dot{Z}(t)} = \underbrace{\begin{pmatrix} 0 & 1 \\ -B(t) & -A(t) \end{pmatrix}}_{M(t)} \underbrace{\begin{pmatrix} X(t) \\ \dot{X}(t) \end{pmatrix}}_{Z(t)}+ \underbrace{\begin{pmatrix} 0 \\ C(t) \end{pmatrix}}_{q(t)}. \]
We work in the space $\leb_2^2(\Omega)$ of $2$-dimensional random vectors whose components belong to $\leb^2(\Omega)$. Given $Z=(Z_1,Z_2)\in \leb_2^2(\Omega)$, its norm is defined as 
\[ \|Z\|_{\leb_2^2(\Omega)}=\max\{\|Z_1\|_{\leb^2(\Omega)}, \|Z_2\|_{\leb^2(\Omega)}\}. \] 
On the other hand, given a random matrix $B=(B_{ij})$, we define the following norm:
\[ |||B|||=\max_i \sum_j \|B_{ij}\|_{\leb^\infty(\Omega)}. \]
In the case of the random matrix $M(t)$, it holds
\begin{equation}
 |||M(t)|||=\max\{1,\|A(t)\|_{\leb^\infty(\Omega)}+\|B(t)\|_{\leb^\infty(\Omega)}\}. 
 \label{normaMt}
\end{equation}
Given $Z,Z'\in \leb_2^2(\Omega)$, we have 
\[ \|(M(t)Z+q(t))-(M(t)Z'+q(t))\|_{\leb_2^2(\Omega)}=\|M(t)(Z-Z')\|_{\leb_2^2(\Omega)}\leq \underbrace{|||M(t)|||}_{k(t)}\cdot \|Z-Z'\|_{\leb_2^2(\Omega)}. \]
Since $A(t)$ and $B(t)$ are continuous stochastic processes in the $\leb^{\infty}(\Omega)$ sense, the real maps
\[t\in(t_0-r,t_0+r)\mapsto \|A(t)\|_{\leb^{\infty}(\Omega)},\quad t\in(t_0-r,t_0+r)\mapsto \|B(t)\|_{\leb^{\infty}(\Omega)}\]
are continuous. By (\ref{normaMt}), the deterministic function $k(t)$ is continuous on $(t_0-r,t_0+r)$. This implies that $k\in\leb^1([t_0-r_1,t_0+r_1])$ for each $0<r_1<r$. By \cite[Th.~5.1.2]{soong}, there is uniqueness of mean square solution for (\ref{problem}) on $[t_0-r_1,t_0+r_1]$. Since $r_1$ is arbitrary, there is uniqueness of solution on $(t_0-r,t_0+r)$.
\end{proof}

\section{Comparison with other methods}\label{sec_comparison}

A final objective of this paper is to relate our method based on \cite{contribucio} (which is based on the deterministic Fr\"{o}benius method) with other well-known techniques to tackle (\ref{problem2}). In \cite{contribucio}, the random power series method was compared, both theoretically and in numerical experiments, with Monte Carlo simulations and the dishonest method \cite{Henderson_Plaschko_2006}. It was demonstrated that Monte Carlo simulations imply a more expensive computational cost to calculate accurately the expectation and variance statistics for $t$ near $t_0$, due to the slow rate of convergence. However, Monte Carlo simulations usually allow validating the numerical results obtained, as they always present convergence with a similar rate for every stochastic system \cite[p.~53]{xiu}.

The article \cite{contribucio} does not compare Fr\"obenius method with generalized polynomial chaos (gPC) expansions \cite{xiu,xiu_article,adapted,adaptedRVT,depen,epidRafa}, although it has been proved to be a powerful technique to deal with general continuous and discrete stochastic systems with absolutely continuous random input coefficients. Due to the spectral mean square convergence of the Galerkin projections, the expectation and variance statistics of the response process can be approximated with small orders of truncation. In the particular setting of random second order linear differential equations, only \cite{airygpc,airygpc2} analyze the application of gPC expansions to Airy's random differential equation, by assuming independence between the random input parameters. Recently, we have also studied the application of gPC expansions to the Legendre random differential equation with statistically dependent inputs in an arXiv preprint \cite{arxiv}. It could be part of a future work the application of gPC expansions to general random second order linear differential equations (\ref{problem2}). We believe that this is important because both the Fr\"obenius method and gPC expansions may validate each other in applications, since they provide good approximations of the expectation and variance statistics rapidly. Moreover, we believe that the gPC approach may provide better approximations of the statistics in the case of large times, see for example \cite{epid_Benito}, where for the classical continuous epidemic models (SIS, SIR, etc.) uncertainty quantification is performed via gPC up to time $60$ with chaoses bases of order just $2$ and $3$, producing very similar results; or \cite{epidRafa}, where an analogous study is performed for the corresponding discrete epidemiological models up to time $30$. Nonetheless, an excessively large number of input parameters may pose problems to the gPC-based method: if the chaos order is $p$ and the degree of uncertainty is $s$, then the length of the basis for the gPC expansions is $(p+s)!/(p!s!)$ \cite{xiu}, which may make the method computationally inviable. Another drawback of the gPC technique is that catastrophic numerical errors usually appear for large chaos orders, specially when dealing with truncated distributions \cite[Example~4.3]{arxiv}, \cite{gpc_rec_nostre}.

In \cite{contribucio}, we did not compare our methodology with the random differential transform method proposed in \cite{benito_transform}. Given a stochastic process $U(t)$, its random differential transform is defined as
\[ \hat{U}(k)=\frac{U^{(k)}(t_0)}{k!}. \]
Its inverse transform is defined as
\[ U(t)=\sum_{k=0}^\infty \hat{U}(k)(t-t_0)^k. \]
Notice that we are actually considering Taylor series in a random calculus setting. It is formally assumed that the series $\sum_{k=0}^\infty \hat{U}(k)(t-t_0)^k$ is mean square convergent on an interval $(t_0-r,t_0+r)$, $r>0$. The computations with the random differential transform method are analyzed in \cite[Th.~2.1]{benito_transform}:
\begin{itemize}
\item[(i)] If $U(t)=F(t)\pm G(t)$, then $\hat{U}(k)=\hat{F}(k)\pm \hat{G}(k)$.
\item[(ii)] If $U(t)=\lambda F(t)$, where $\lambda$ is a random variable, then $\hat{U}(k)=\lambda \hat{F}(k)$.
\item[(iii)] If $U(t)=G^{(m)}(t)$, then $\hat{U}(k)=(k+1)\cdots (k+m)\hat{G}(k+m)$.
\item[(iv)] If $U(t)=F(t)G(t)$, then $\hat{U}(k)=\sum_{n=0}^k \hat{F}(n)\hat{G}(k-n)$.
\end{itemize}
Notice that (iii) and (iv) can be seen as consequences of differentiating random power series \cite[Th.~3.1]{contribucio} and multiplying random power series \cite[Th.~3.2]{contribucio}, respectively. Thereby, the random transform method is actually the random Fr\"obenius method. The recursive equations found for $\hat{X}(k)$ are (\ref{xn2}). Our Theorem~\ref{nostre}, Theorem~\ref{nostre2} and Theorem~\ref{nostre3} give the conditions under which the inverse transform $\sum_{k=0}^\infty \hat{X}(k)(t-t_0)^k$ converges.

Thus, we believe that our recent contribution \cite{contribucio} together with the notes presented in this paper give an excellent approach to tackle (\ref{problem}) and/or (\ref{problem2}) with analytic random input processes. Apart from obtaining a mean square analytic solution to (\ref{problem}) and/or (\ref{problem2}), the expectation and variance of it can be calculated for uncertainty quantification.

\section{Summary, conclusions and future lines of research}\label{sec_conclusiones}

In this paper, we have written some notes and comments to complete our recent contribution \cite{contribucio} on the random non-autonomous second order linear differential equation. The main theorem from \cite{contribucio}, which deals with the homogeneous case, has been restated in a more convenient form to deal with practical applications. We have addressed the non-homogeneous case, by proving an existence theorem of mean square solution and performing a numerical example. On the other hand, the uniqueness of solution has been established by using Picard Theorem for mean square calculus. A comparison of the extant techniques for uncertainty quantification (Monte Carlo, gPC expansions, random differential transform method) with respect to the random Fr\"obenius method has been studied. 

This paper is a contribution to the field of random differential equations, as it completely generalizes to a random framework the deterministic theory on second order linear differential equations with analytic input data. To carry out the study, mean square calculus, and in general $\leb^p$ random calculus, become powerful tools to establish the theoretical results and perform uncertainty quantification.

Some future research lines related to the contents of this paper are the following: 
\begin{itemize}
\item Solve the open problem raised in this paper at the end of Section~\ref{homoSection}, concerning the necessity of the hypotheses of Theorem~\ref{nostre2}.
\item Apply the technique of gPC expansions and stochastic Galerkin projections to general random second order linear differential equations.
\item Extend Theorem~\ref{nostre3} to higher order random linear differential equations. Probably, one would need to require all input stochastic processes to be random power series in an $\leb^{\infty}$ sense, in analogy to the hypotheses of Theorem~\ref{nostre3}.
\item Apply the random Fr\"obenius method to the random Riccati differential equation with analytic input processes. In \cite[Section~3]{benito_transform}, the authors applied the random differential transform method (which is equivalent to a formal random Fr\"obenius method) to a particular case of the random Riccati differential equation with a random autonomous coefficient term. It would be interesting to apply the random Fr\"obenius method in the situation in which all input coefficients are analytic stochastic processes, by proving theoretical results and performing numerical experiments.
\end{itemize}

\section*{Acknowledgements}
This work has been supported by the Spanish Ministerio de Econom\'{i}a y Competitividad grant MTM2017--89664--P. The author Marc Jornet acknowledges the doctorate scholarship granted by Programa de Ayudas de Investigaci\'on y Desarrollo (PAID), Universitat Polit\`ecnica de Val\`encia.

\section*{Conflict of Interest Statement} 
The authors declare that there is no conflict of interests regarding the publication of this article.

\end{document}